\documentclass[
final
]{dmtcs-episciences}

\usepackage[utf8]{inputenc}
\usepackage{subfigure}
\usepackage{graphicx}
\usepackage{pstricks}
\usepackage{pstricks,pst-node}

\newtheorem{teo}{Theorem}
\newtheorem{coro}{Corollary}
\newtheorem{lem}{Lemma}
\newtheorem{rem}{Remark}
\newtheorem{claim}{Claim}
\newenvironment{proof}{{\bf Proof\ }}{\hfill $\BBox$ \smallskip}
\newenvironment{proofi}{{\bf {Proof of Claim 1.\ }}}{\hfill $\Box$ \smallskip}
\usepackage[round]{natbib}

\author{Hortensia Galeana-S\'{a}nchez\affiliationmark{1}\thanks{Research supported by CONACyT, M\'exico under proyect $219840$ and UNAM-DGAPA-PAPIIT IN104717.}
  \and Mika Olsen\affiliationmark{2}\thanks{Research supported by CONACyT, M\'exico under proyect $222104$.}}
\title[Solving the kernel perfect problem for digraphs in some families of generalized tournaments]{Solving the kernel perfect problem by (simple) forbidden subdigraphs for digraphs in some families of generalized tournaments and generalized bipartite tournaments}
\affiliation{
  UNAM, M\'exico\\ 
  UAM-Cuajimalpa, M\'exico}
\keywords{kernel, perfect graph, kernel perfect digraph, locally in-/out-semicomplete digraph, asymmetric arc-locally in-/out-semicomplete digraph, asymmetric $3$-(anti)-quasi-transitive digraph}
\received{2017-12-04}
\revised{2018-7-5, 2018-11-5}
\accepted{2018-11-6}
\begin{document}
\publicationdetails{20}{2018}{2}{16}{4122}
\maketitle
\begin{abstract}
A digraph such that every proper induced subdigraph has a kernel is said to be \emph{kernel perfect} (KP for short) (\emph{critical kernel imperfect} (CKI for short) resp.) if the digraph has  a kernel (does not have a kernel resp.).
The unique CKI-tournament is $\overrightarrow{C}_3$ and the unique KP-tournaments are the transitive tournaments, however bipartite tournaments are KP. 
In this paper we characterize the CKI- and KP-digraphs for the following families of digraphs: locally in-/out-semicomplete, asymmetric arc-locally in-/out-semicomplete, asymmetric $3$-quasi-transitive and asymmetric $3$-anti-quasi-transitive $TT_3$-free and we state that the problem of determining whether a digraph of one of these families is CKI is polynomial, giving a solution to a problem closely related to the following conjecture posted by Bang-Jensen in 1998: the kernel problem is polynomially solvable for locally in-semicomplete digraphs. 
\end{abstract}

%\emph{MSC 2000: }{05C20, 05C69, 05C75}
%\end{keyword}

%\end{frontmatter}

\section{Introduction}
A family which is a generalization of tournaments is a family of digraphs that in some way preserves basic structures of the tournaments, an interesting survey of generalizations of tournaments can be found in \cite{bang1998generalizations}. Generalizations of tournaments have been widely studied, more than 300 papers have been published in this topic and this has improved the understanding of topics such as hamiltonicity, domination and pancyclicity, properties that digraphs of  some families of generalized tournaments preserve, see \cite{bang1998generalizations, bang1995quasi}. The locally semicomplete digraphs introduced %by Bang-Jensen in 1991 
in \cite{bang1990locally} are among the families which have been most studied. Locally in-semicomplete digraphs and locally out-semicomplete digraphs are also generalizations of tournaments.
A digraph $D$ is \emph{locally in-semicomplete} (resp. \emph{locally out-semicomplete}) if for any vertex  $v\in V(D)$, the in-neighborhood (resp. out-neighborhood) induces a semicomplete digraph in $D$. A digraph $D$ is \emph{locally semicomplete}  if it is both locally in-semicomplete and locally out-semicomplete. Observe that locally semicomplete digraphs are locally in-semicomplete and locally out-semicomplete, but the converse is not true (see Figure \ref{Fig locally-in}), 
\begin{figure}[h!]
\centering
\begin{pspicture}(4.5,1.4)
        \psset{unit=.6, nodesep=3pt}     
        \cnode*(4.7,0){1.7pt}{10}
        \cnode*(4.7,2){1.7pt}{11} 
        \cnode*(.5,0){1.7pt}{00}
        \cnode*(.5,2){1.7pt}{01}
        \cnode*(2.5,1){1.7pt}{1}
        \ncline{->}{00}{01}
        \ncline{->}{00}{1}
        \ncline{->}{01}{1}
        \ncline{->}{1}{10}
        \ncline{->}{1}{11}
\end{pspicture}
\caption{A locally-in semicomplete digraph which is not a locally semicomplete digraph.}
\label{Fig locally-in}
\end{figure}
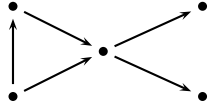
the class of locally in-semicomplete digraphs (resp. {locally out-semicomplete}) is a quite wider class than the class of locally semicomplete digraphs.
The  arc-locally semicomplete digraphs was defined in \cite{bang2004structure}, this definition is somehow close related to the definition of locally semicomplete digraphs, although they are generalizations of bipartite tournaments and it is surprising so few results there are on this family, properties such as hamiltonicity and pancyclism are known.
A family of digraphs is a \emph{family of generalized bipartite tournaments} if the digraphs preserve in some way basic structures of bipartite tournaments.
The arc-locally in-semicomplete, arc-locally out-semicomplete, $3$-quasi-transitive and $3$-anti-quasi-transitive digraphs are families of generalized bipartite tournaments.
A digraph $D$ is \emph{arc-locally in-semicomplete} (\emph{arc-locally out-semicomplete}, resp.) if, for every arc $uv\in D$ and every pair of vertices $x,y$ such that $x\in N^-(u)$ and $y\in N^-(v)$ ($x\in N^+(u)$ and $y\in N^+(v)$, resp.),  $x,y$ are adjacent. A digraph $D$ is \emph{arc-locally semicomplete} if $D$ is arc-locally in-semicomplete and arc-locally out-semicomplete.
%A digraph $D$ is \emph{$3$-quasi-transitive} (\emph{$3$-anti-quasi-transitive} resp.) \cite{bang2004structure} if, for every directed (anti-directed resp.) $4$-path $uvwx$, $u$ and $x$ are adjacent.
In \cite{bang2004structure}, a digraph $D$ was defined as \emph{$3$-quasi-transitive} (\emph{$3$-anti-quasi-transitive} resp.) if, for every directed (anti-directed resp.) $4$-path $uvwx$, $u$ and $x$ are adjacent.

Kernels are an important topic in the theory of digraphs, they where introduced in \cite{neumann1944theory} and has received lot of attention due to its theoretical interest as well as its many applications in different areas such as game theory, argumentation theory, logic, logic programming and artificial intelligence, see \cite{le2000counterexamples, creignou1995class, dimopoulos1996graph, dung1995acceptability}. A \emph{kernel} in a digraph is an independent and absorbent set of vertices (a subset of vertices $S$ of a digraph $D$ is an absorbent set if for every vertex $u\in V(D)\setminus S$ there is a vertex $v\in S$ such that $uv\in A(D)$), for surveys of kernels see \cite{bang1998generalizations, haynes1998domination}. The problem of deciding  whether a digraph D has a kernel is $\mathcal{NP}$-complete, see \cite{chvatal1973computational}, even for planar digraphs with in-degree and out-degree at least 1 and degree at most three, see \cite{fraenkel1981planar}.
Due to the difficulty of this topic, the study of kernels is centered in sufficient conditions to asure the existence of kernels and the study of fixed classes of digraphs having a kernel or fixed classes of kernel perfect digraphs. The existence of  kernels in digraphs with a perfect underlying graph has been studied in \cite{galeanarojas2006kernels}; the existence of kernels in locally in-/out-semicomplete digraphs has been studied in \cite{galeana1995b1, galeana1997characterization};  and the existence of kernels in arc-locally in-/out-semicomplete digraphs has been studied in \cite{galeana2006kernels}.
The Strong Perfect Graph Conjecture,  stated by C. Berge in 1960 and proved in \cite{chudnovsky2006strong} states that a graph $G$ is perfect if and only if $G$ contains neither $C_{2n+1}$ nor the complement of $C_{2n+1}$, $n\geq 2$, as an induced subgraph.
Many authors have contributed to obtain nice properties and interesting characterizations of Perfect Graphs, see \cite{chvatal1984topics,ramirez2001perfect}. 
The \emph{underlying graph} $G_D$ of a digraph $D$ is the graph on the vertex set $V(G_D)=V(D)$ whith $\{u,v\}\in E(G_D)$ if and only if $\{uv,vu\}\cap A(D)\neq \emptyset$.
The digraphs of some families of generalized tournaments have a perfect graph as underlying graph, for instance \emph{quasi-transitive digraphs} (for every directed path $uvw$ there is an arc between $u$ and $w$) and \emph{semicomplete multipartite tournaments} (the underlying graph is a complete multipartite graph).

A digraph such that every proper induced subdigraph has a kernel is said to be a \emph{kernel perfect digraph} (KP-digraph) (\emph{critical kernel imperfect digraph} (CKI-digraph) resp.) if the digraph $D$ has  a kernel (does not have a kernel resp.).
In  \cite{berge1990recent} it was conjectured that a graph $G$ is perfect if and only if any orientation by sinks of $G$ is a kernel perfect digraph, the authors considered orientations of $G$ by directing each edge of $G$ in at least one of the two possible directions. An orientation of $G$ is an \emph{orientation by sinks} (or normal) if every induced semicomplete subgraph $H$ of $G$ has an absorbing vertex in $H$ (a vertex $v\in V(H)$ such that $uv\in A(H)$ for every vertex $u\in V(H)\setminus \{v\}$).
This conjecture was proved in \cite{berge1990recent} and in \cite{boros1996perfect}, and it constructs an important bridge between two topics in graph theory: namely colorings and kernels.
It is important to stress that a digraph $D$ without induced CKI-digraph is a KP-digraph and so, it does have a kernel. Although there are digraphs with kernels which are not KP.  Hence, another tool to decide whether a digraph has a kernel is characterizing the CKI-digraphs; due to a result of the first author and Neumann-Lara \cite{galeana1991extending} CKI-digraphs cannot be characterized by forbidden minors. For structural properties of CKI-digraphs see \cite{balbuena2014structural, galeana2016some, galeana1991extending, galeana2016cki} and for characterizations of families of CKI-digraphs see \cite{galeana2016cki,galeana2016characterization}.
The unique tournament which is CKI is $\overrightarrow{C}_3$ and the unique KP-tournaments are the transitive tournaments, however any induced subdigraph of a bipartite tournament does have a kernel. Hence, bipartite tournaments are KP. In \cite{galeana1986kernel} it was proved that the semicomplete CKI-digraphs are $\overrightarrow{C}_3$ and the family $D\cong \overrightarrow{C}_n(1,\pm2,\pm3,\dots,\pm\lfloor \frac{n}{2}\rfloor)$, for some $n\ge4$, where  $\overrightarrow{C}_m(J)$ is a \emph{circulant }(or \emph{rotational) digraph} defined by $V(\overrightarrow{C}_m(J))=\mathbb{Z}_m$ and $A(\overrightarrow{C}_m(J))=\{ (i,j):i,j\in \mathbb{Z}_m,j-i\in J\} $, with  $\mathbb{Z}_m$ as the cyclic group of integers modulo $m$ $(m\geq 2)$ and $J$ is a nonempty subset of $\mathbb{Z}_m\setminus \{0\}$.
 In \cite{galeana2016characterization} the authors characterized the locally semicomplete CKI-digraphs as directed odd cycles, $\overrightarrow{C}_7(1,2)$ and $D\cong \overrightarrow{C}_{m}(1,\pm 2,\pm 3,\dots,\pm \lfloor \frac{n}{2} \rfloor )$, for $n\ge4$. 

We  characterize the CKI-digraphs  having a perfect graph as underlying graph, using the relation between perfect graphs and kernel perfect graphs \cite{berge1990recent,boros1996perfect, galeana2012new} and we characterize the locally in-semicomplete digraphs and the locally out-semicomplete digraphs which are CKI-/KP-digraphs. It is important to stress that the property of being CKI is not preserved for the converse digraph of a CKI digraph, see \cite{duchet1981note}, the \emph{converse} digraph $D^{-1}$ of a digraph $D$ is obtained by reversing the arcs of $D$. Hence, characterizing the locally in-semicomplete digraphs and the locally out-semicomplete digraphs which are  CKI-/KP-digraphs is not the same problem, 
although for these two families the families of CKI-/KP-digraphs remains the same. 
Finally, we characterize the asymmetric CKI-/KP-digraphs which are arc-locally in-semicomplete, arc-locally out-semicomplete or $3$-anti-quasi-transitive $TT_3$-free digraphs as the directed odd cycles and the asymmetric $3$-quasi-transitive CKI-digraphs as $\overrightarrow{C}_3$.
Moreover, we state that the problem of determining whether a digraph is CKI is polynomial for digraphs of the following families of digraphs: locally in-/out-semicomplete, asymmetric arc-locally in-/out-semicomplete, $3$-quasi-transitive and asymmetric $3$-anti-quasi-transitive $TT_3$-free. Hence, we give a solution to a problem closely related to the following conjecture posted  in \cite{bang1998generalizations}: the kernel problem is polynomially solvable for locally in-semicomplete digraphs.

%%%%%%%%%%%%%%%%%%%%%%%%%%%%%%%%

\section{Definitions and prelimiaries}
For general concepts and notation we refer the reader to \cite{bang2008digraphs}.
The paths and cycles considered in this paper are not necessarily directed paths or cycles.
We denote the path $P$  by the sequence of its vertices $P=u_0u_1\dots u_n$. We say that $P$ is a \emph{directed path}  if $u_iu_{i+1}\in A(D)$ for $0\le i\le n-1$, $P$ is an \emph{anti-directed path}  if it has no directed subpath of length $2$.
Let $D$ be a digraph and $H$ a proper subdigraph of $D$. An arc $uv\in A(D)\setminus A(H)$ is a \emph{diagonal} of $H$ whenever $u,v\in V(H)$.
For a subset of vertices $S$ of a digraph $D$, $D[S]$ denotes the digraph induced by the vertex set $S$.
\begin{rem}\label{CKI}
Directed odd cycles and $\overrightarrow{C}_7(1,2)$ are CKI-digraphs.  Directed even cycles are KP-digraphs.
\end{rem}
The following remark is a consequence of the definition of KP-digraphs and CKI-digraphs.
\begin{rem}\label{sub CKI}\label{subNIC}
If $D$ is a CKI-digraph (or a KP-digraph), then $D$ has no proper induced CKI-subdigraph. In particular, $D$ has no proper induced subdigraph isomorphic to  a directed odd cycle.
\end{rem}

Circulant digraphs are regular, vertex transitive and isomorphic to its converse digraph (the \emph{converse} digraph $D^{-1}$ of a digraph $D$ is obtained by reversing the arcs of $D$). Thus, if a circulant digraph $D$ is CKI, then $D^{-1}$ is also CKI. This is not true in general, due to  \cite{duchet1981note}.
A graph is a \emph{perfect graph}, if for every induced subgraph, the clique number equals the chromatic number.

\begin{teo}[Strong Perfect Graph Theorem; \cite{chudnovsky2006strong}]\label{T Chud}
A graph $G$ is not perfect if and only if $G$ has as an induced subgraph
\begin{enumerate}
\item[$(i)$]  an odd cycle on at least $5$ vertices or
\item[$(ii)$]  the complement of an odd cycle on at least $7$ vertices.
\end{enumerate}
\end{teo}

Let $H$ be a subdigraph of $G$. An \emph{absorbing vertex} in $H$ is a vertex  $v\in V(H)$ such that $uv\in A(H)$ for every vertex $u\in V(H)\setminus \{v\}$. A digraph $D_G$ is an orientation (or biorientation) of a graph $G$ if $V(D_G)=V(G)$ and $\{u,v\}\in E(G)$ if and only if $\{uv,vu\}\cap A(D)\neq \emptyset$, and an orientation $D_G$ of $G$ is an \emph{orientation by sinks} if every complete subgraph of $G$ has an absorbing vertex in $D_G$.

\begin{teo}[\cite{berge1990recent,boros1996perfect}]\label{perf s order}
A graph $G$ is perfect if and only if every orientation by sinks of $G$ is a KP-digraph.
\end{teo}

For another relation between kernels and perfect graphs see \cite{galeana2012new}.

\begin{teo}[\cite{galeana1986kernel}]\label{T semi}
A~ semicomplete~ digraph~ $D$~ is~ a~ CKI-digraph~ if~ and~ only~ if~ $D\cong \overrightarrow{C}_3$ or  ${D\cong \overrightarrow{C}_n(1,\pm2,\pm3,\dots,\pm\lfloor \frac{n}{2}\rfloor)}$, for some $n\ge4$.
\end{teo}

\begin{teo}[\cite{galeana2016characterization}]\label{T local semi}
A locally semicomplete digraph $D$ is a CKI-digraph if and only if  $D$ is an odd cycle, $D\cong \overrightarrow{C}_7(1,2)$ or $D\cong \overrightarrow{C}_{n}(1,\pm 2,\pm 3,\dots,\pm \lfloor \frac{n}{2} \rfloor )$, for some $n\ge4$.
\end{teo}

\section{Generalized tournaments}

In this section we characterize the CKI-digraphs with a perfect underlying graph and the locally in- and the locally out-semicomplete CKI-digraphs.
As a consequence of Remark \ref{sub CKI}, Theorems \ref{perf s order} and  \ref{T semi} we have the following characterization of CKI-digraphs having a perfect underlying graph.
\begin{teo}\label{*}
Let $D$ be a digraph such that the underlying graph, $G_D$, is a perfect graph. Then $D$ is CKI if and only if $D\cong \overrightarrow{C}_3$ or $D\cong \overrightarrow{C}_n(1,\pm2,\pm3,\dots,\pm\lfloor \frac{n}{2}\rfloor)$, for some $n\ge4$.
\end{teo}

\begin{proof}
Let $D$ be a CKI-digraph such that $G_D$ is a perfect graph. By Theorem \ref{perf s order}, $D$ is not oriented by sinks, and therefore $D$ has an induced semicomplete subdigraph $H$ with no sink. Hence, $H$ is not KP and contains an induced subdigraph $H'$ which is semicomplete and CKI. By Remark \ref{sub CKI}, $D$ has no proper induced CKI-subdigraph, therefore $D=H'$. Hence, $G_D$ is complete,  $D$ is semicomplete and by Theorem \ref{T semi}, $D\cong \overrightarrow{C}_3$ or $D\cong \overrightarrow{C}_n(1,\pm2,\pm3,\dots,\pm\lfloor \frac{n}{2}\rfloor)$, for some $n\ge4$.
\end{proof}

A digraph is a \emph{semicomplete multipartite tournament} if the underlying graph is a complete multipartite graph. A digraph is \emph{quasi-transitive} if for every directed path $uvw$ there is an arc between $u$ and $w$.  We have the following corollary.

\begin{coro}\label{C semi quasi}
Let $D$ be a semicomplete multipartite digraph or a quasi-transitive digraph. Then $D$ is  CKI if and only if  $D\cong \overrightarrow{C}_3$ or
$D\cong\overrightarrow{C}_{m}(1,\pm 2,\pm 3,\dots,\pm \lfloor \frac{m}{2} \rfloor )$ for some $m\ge4$.
\end{coro}

\begin{proof}
The underlying graph of a semicomplete multipartite digraph is a complete multipartite graph, which is a perfect graph  and
the underlying graph of a quasi-transitive digraph is a comparability graph, see  \cite{duchet1984introduction}, which is a perfect graph.
Hence, the result follows.
\end{proof}

As a consequence of Theorem \ref{*} and Corollary \ref{C semi quasi} we have the following result.
\begin{teo}\label{*KP}
Let $D$ be a digraph such that the underlying graph %, $G_D$, 
is a perfect graph. Then $D$ is KP if and only if $D$ has no induced subdigraph $H$ such that $H\cong\overrightarrow{C}_3$ or $H\cong \overrightarrow{C}_n(1,\pm2,\pm3,\dots,\pm\lfloor \frac{n}{2}\rfloor)$, for some $n\ge4$. In particular, semicomplete multipartite digraphs and  quasi-transitive digraphs are KP if and only if they have no induced subdigraph $H$ with $H\cong \overrightarrow{C}_3$ or
$H\cong\overrightarrow{C}_{m}(1,\pm 2,\pm 3,\dots,\pm \lfloor \frac{m}{2} \rfloor )$, for some $m\ge4$.
\end{teo}

%%%%%%%%%%%%%%%%%%%%

The following lemma determines the orientation of some induced subdigraphs in asymmetric locally in-semicomplete digraphs or asymmetric locally out-semicomplete digraphs. 
An arc ${uv\in A(D)}$ is \emph{asymmetric} if $vu\notin A(D)$ (\emph{symmetric} if $vu\in A(D)$ resp.). A digraph is \emph{asymmetric} if all its arcs are asymmetric arcs.
Observe that if $n\ge 5$, $\overline{G}_D$ is an induced cycle on $n$ vertices if and only if $G_D$ is an antihole on $n$ vertices.

\begin{lem}\label{semi ciclo}
Let $D$ be a locally in-semicomplete digraph or a locally out-semicomplete digraph. If $G_D$ is an induced cycle on at least $4$ vertices, then $D$ is a directed cycle. If $\overline{G}_D$ is an induced cycle on at least $5$ vertices, then $D\cong \overrightarrow{C}_{2n+1}(2,-3,4,-5,\dots, (-1)^nn)$.
\end{lem}

\begin{proof}
Let $D$ be a locally in-semicomplete digraph (locally out-semicomplete digraph resp.). 
Suppose that $G_D$ is an induced cycle $C=u_0u_1\dots u_{n-1}u_0$, with $n\ge4$. Since $C$ is induced and $D$ is  locally in-semicomplete, (locally out-semicomplete resp.), 
there is no vertex $u_i\in V(C)$ such that $u_{i-1},u_{i+1}\in N^-(u_i)$, ($u_{i-1},u_{i+1}\in N^+(u_i)$ resp.), 
with indices taken modulo $n$. Thus, $C$ is an induced directed  cycle on at least $4$ vertices. Hence, if $D$ is a locally in-/out-semicomplete digraph and $G_D$  is an induced cycle on at least $n$ vertices, $n\ge4$, then $D$ is a induced directed cycle on $n$ vertices.

If $D$ is a locally in-semicomplete digraph or a locally out-semicomplete digraph and $\overline{G}_D$ is an induced cycle on $5$ vertices, then ${G}_D$ is an induced cycle on $5$ vertices and $D$ is a directed cycle on $5$ vertices. 

Let $D$ be a locally in-semicomplete digraph. Suppose that $\overline{G}_D$ is an induced cycle $C=u_0u_1\dots u_{n-1}u_0$, with $n\ge6$. Observe that $u_i$ and $u_{i+1}$ are not adjacent in $D$ for $0\le i\le n-1$ with indices taken modulo $n$ and since $D$ is locally in-semicomplete, $N^+(u_i)\cap N^+(u_{i+1})=\emptyset$. 

Let $u_iu_{i+j}\in A(D)$. Clearly, $2\le j\le n-2$. If $j\ge3$ and $u_{i+1}u_{i+j}\in A(D)$, then $u_i,u_{i+1}\in N^-(u_{i+j})$, a contradiction because $\{u_i,u_{i+1}\}\notin E(G_D)$. Thus, if $j\ge3$, it follows that $u_{i+1}u_{i+j}\notin A(D)$, see Figure \ref{Fig}. 
\begin{figure}[h!]
\centering
\begin{pspicture}(4,2.2)
        \psset{unit=.8, nodesep=3pt}        
        \rput(4.75,2.65){\rnode{1}{$u_{i-1}$}}
        \rput(5.1,1.3){\rnode{2}{$u_i$}}
        \rput(4.75,0){\rnode{3}{$u_{i+1}$}}
        \rput(.3,0){\rnode{6}{$u_{i+j-1}$}}
        \rput(0,1.3){\rnode{6}{$u_{i+j}$}}
        \rput(.3,2.65){\rnode{6}{$u_{i+j+1}$}}
        \cnode*(4.700,1.3){1.7pt}{i}
        \cnode*(4.172,2.67){1.7pt}{i-}
        \cnode*(4.172,0.03){1.7pt}{i+}
        \cnode*(1.046,0.03){1.7pt}{j-}
        \cnode*(0.518,1.3){1.7pt}{j}
        \cnode*(1.046,2.67){1.7pt}{j+}
        \ncarc[arcangle=10, linestyle=dotted,dotsep=2pt]{-}{i-}{i}
        \ncarc[arcangle=10, linestyle=dotted,dotsep=2pt]{-}{i}{i+}
        \ncline{->}{i}{j}
        \ncarc[arcangle=-25, linestyle=dotted,dotsep=2pt]{->}{i-}{j}
        \ncarc[arcangle=25, linestyle=dotted,dotsep=2pt]{->}{i+}{j}
        \ncarc[arcangle=10, linestyle=dotted,dotsep=2pt]{-}{j-}{j}
        \ncarc[arcangle=10, linestyle=dotted,dotsep=2pt]{-}{j}{j+}
        \ncarc[arcangle=-2]{->}{j}{i-}
        \ncarc[arcangle=2]{->}{j}{i+}
\end{pspicture}
\caption{The dotted arcs are not arcs of $D$.}
\label{Fig}
\end{figure}
Analogously, if $j\le n-3$, then $u_{i-1}u_{i+j}\notin A(D)$, see Figure \ref{Fig}. Since $C$ is an induced cycle in  $\overline{G}_D$, it follows that for  $u_iu_{i+j}\in A(D)$ 
\begin{equation}\label{i->i+j}
(a)\mbox{ if }j\ge3\mbox{, then }u_{i+j}u_{i+1}\in A(D),\qquad (b)\mbox{ if }j\le n-3\mbox{, then }u_{i+j}u_{i-1}\in A(D).
\end{equation}

First we prove that $D$ is  asymmetric, and then we determine the arcs of $D$.

For a contradiction, suppose that $\{u_i,u_{i+j}\}$ is a symmetric arc for some $j\ge 4$.
Since $u_iu_{i+j}\in A(D)$  it follows by \ref{i->i+j}$(a)$ that $u_{i+j}u_{i+1}\in A(D)$, see Figure \ref{Fig2}; since $u_{i+j}u_i\in A(D)$ and $n-((i+j)-i)\le n-4$, it follows by \ref{i->i+j}$(b)$ that $u_{i}u_{i+j-1}\in A(D)$, see Figure \ref{Fig2}. 
In this case $u_{i+j-1}u_{i+1}\notin A(D)$, because $D$ is locally in-semicomplete and $u_{i+j-1}$ and $u_{i+j}$ are not adjacent; and $u_{i+1}u_{i+j-1}\notin A(D)$, because $u_i$ and $u_{i+1}$ are not adjacent. 
Thus,  $\{u_{i+j-1},u_{i+1}\}\notin E(G_D)$, see Figure \ref{Fig2},
\begin{figure}[h!]
\centering
\begin{pspicture}(4.5,1.4)
        \psset{unit=.8, nodesep=3pt}     
        \rput(4.95,1.35){\rnode{1}{$u_{i+j-1}$}}
        \rput(5.3,0){\rnode{2}{$u_{i+j}$}}
        \rput(.45,1.35){\rnode{1}{$u_{i+1}$}}
        \rput(.15,0){\rnode{6}{$u_{i}$}}
        \cnode*(4.700,0){1.7pt}{i}
        \cnode*(4.172,1.37){1.7pt}{i-}
        \cnode*(.518,0){1.7pt}{j}
        \cnode*(1.046,1.37){1.7pt}{i+}
        \ncarc[arcangle=10, linestyle=dotted,dotsep=2pt]{-}{i-}{i}
        \ncarc[arcangle=-10]{->}{i}{i+}
        \ncarc[arcangle=10]{->}{i}{j}
        \ncarc[arcangle=-10]{<-}{i}{j}
        \ncarc[arcangle=8]{->}{j}{i-}
        \ncarc[arcangle=8,  linestyle=dotted,dotsep=2pt]{-}{j}{i+}
        \ncarc[arcangle=-15, linestyle=dashed,dash=5pt 2pt]{-}{i-}{i+}
\end{pspicture}
\caption{The dotted arcs are not arcs of $D$.}\label{Fig2}
\end{figure}
a contradiction because $j\ge 4$ and  $C$ is an induced cycle in  $\overline{G}_D$.
Thus, if $j\ge 4$, then $\{u_i,u_{i+j}\}$ is an asymmetric arc.
By a dual argumentation, using first \ref{i->i+j}$(b)$ and then \ref{i->i+j}$(a)$, we obtain that if $j\le n-4$, then $\{u_i,u_{i+j}\}$ is an asymmetric arc.

If $|V(D)|\ge7$, then every edge $\{u_i,u_{i+j}\}\in E(\overline{G}_D)$ is an asymmetric arc in $D$ because $j=(i+j)-i\ge4$ or $n-j=n-((i+j)-i)\ge4$, that is at least one of the induced paths in $\overline{G}_D$  $u_0u_1\dots u_j$ or $u_ju_{j+1}\dots u_0$ has $4$ arcs. Hence, $D$ is asymmetric if $|V(D)|\ge7$.

If $|V(D)|=6$, it follows that $\{u_i,u_{i+j}\}$ is asymmetric if $j=2,4$. 
For a contradiction, suppose that $\{u_0,u_3\}$ is an asymmetric arc. By \ref{i->i+j}$(a)$, $u_0u_2,u_0u_4,u_3u_1,u_3u_5\in A(D)$. In this case $u_{2}u_{5}\notin A(D)$, because $D$ is locally in-semicomplete, $u_3u_5\in A(D)$ and $u_{2}$ and $u_{3}$ are not adjacent in $D$; and $u_{5}u_{2}\notin A(D)$, because $u_0u_2\in A(D)$ and $u_0$ and $u_{5}$ are not adjacent in $D$. 
Thus,  $\{u_{2},u_{5}\}\notin E(G_D)$, a contradiction. 
Hence, $D$ is asymmetric.

In order to prove that $D\cong \overrightarrow{C}_{2m+1}(2,-3,4,-5,\dots, (-1)^mm)$, we need the following claim. 

\begin{claim}\label{Claim}
Let $j$ be an integer, with $2\le j\le n-2$. If  $u_ku_{k+j}\in A(D)$ for some vertex $u_k\in V(D)$, then  $u_iu_{i+j}\in A(D)$ for every vertex $u_i\in V(D)$.
\end{claim}

\begin{proofi}
Let $k,j$ be integers such that $u_ku_{k+j}\in A(D)$.
If $j=2,3$, then 
by $\ref{i->i+j}(b)$, it follows that $u_{k+j}u_{k-1}\in A(D)$, see Figure \ref{Figbis}$(b)$. 
Consider the arc $u_{k+j}u_{k-1}$, by $\ref{i->i+j}(b)$, it follows that $u_{k-1}u_{k-1+j}\in A(D)$. 
Consider the arc $u_{k-1}u_{k-1+j}$, by $\ref{i->i+j}(b)$, it follows that $u_{k-1+j}u_{k-2}\in A(D)$. 
Consider the arc $u_{k-1+j}u_{k-2}$, by $\ref{i->i+j}(b)$, it follows that $u_{k-2}u_{k-2+j}\in A(D)$. Continuing these two steps it follows that for $j=2,3$, if  $u_ku_{k+j}\in A(D)$ for some vertex $u_k\in V(D)$, then  $u_iu_{i+j}\in A(D)$ for every vertex $u_i\in V(D)$.
Analogously, using $\ref{i->i+j}(a)$, if $j=n-2,n-3$ and $u_ku_{k+j}\in A(D)$, then $u_iu_{i+j}\in A(D)$  for every  $u_i\in V(D)$.

\begin{figure}[h!]
\centering
\begin{pspicture}(10,2)
        \psset{unit=.8, nodesep=3pt}        
        \rput(2.6,0){\rnode{6}{$(a)~ j\ge 3$}}
        \rput(5,.6){\rnode{1}{$u_{k+1+j}$}}
        \rput(5.3,1.95){\rnode{2}{$u_{k+j}$}}
        \rput(.688,.6){\rnode{1}{$u_{k}$}}
        \rput(0,1.95){\rnode{6}{$u_{k+1}$}}
        \cnode*(4.700,1.97){1.7pt}{i}
        \cnode*(4.172,.6){1.7pt}{i-}
        \cnode*(.518,1.97){1.7pt}{j}
        \cnode*(1.046,.6){1.7pt}{i+}
        \ncarc[arcangle=-10, linestyle=dotted,dotsep=2pt]{-}{i-}{i}
        \ncarc[arcangle=10]{<-}{i}{i+}
        \ncarc[arcangle=-10]{->}{i}{j}
        \ncarc[arcangle=-8]{->}{j}{i-}
        \ncarc[arcangle=-8,  linestyle=dotted,dotsep=2pt]{-}{j}{i+}

        \rput(9.6,0){\rnode{6}{$(b)~ j\le n-3$}}
        \rput(12,1.95){\rnode{1}{$u_{k-1+j}$}}
        \rput(12.3,.6){\rnode{2}{$u_{k+j}$}}
        \rput(7.65,1.95){\rnode{1}{$u_{k}$}}
        \rput(7,.6){\rnode{6}{$u_{k-1}$}}
        \cnode*(11.700,.6){1.7pt}{i}
        \cnode*(11.172,1.97){1.7pt}{i-}
        \cnode*(7.518,.6){1.7pt}{j}
        \cnode*(8.046,1.97){1.7pt}{i+}
        \ncarc[arcangle=10, linestyle=dotted,dotsep=2pt]{-}{i-}{i}
        \ncarc[arcangle=-10]{<-}{i}{i+}
        \ncarc[arcangle=8]{->}{j}{i-}
        \ncarc[arcangle=8,  linestyle=dotted,dotsep=2pt]{-}{j}{i+}
        \ncarc[arcangle=15]{->}{i}{j}
        
\end{pspicture}
\caption{The dotted arcs are not arcs of $D$.}\label{Figbis}
\end{figure}
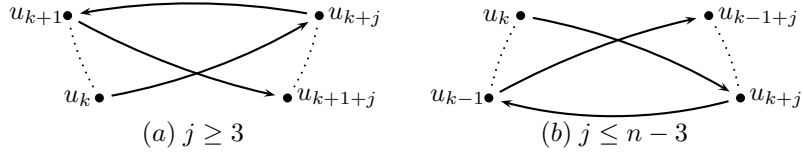

If $3< j< n-3$, then 
by $\ref{i->i+j}(a)$, it follows that $u_{k+j}u_{k+1}\in A(D)$, see Figure  \ref{Figbis}$(a)$.
Consider the arc $u_{k+j}u_{k+1}$, by $\ref{i->i+j}(a)$, it follows that $u_{k+1}u_{k+1+j}\in A(D)$.
Consider the arc $u_{k+1}u_{k+1+j}$, by $\ref{i->i+j}(a)$, it follows that $u_{k+1+j}u_{k+2}\in A(D)$.
Consider the arc $u_{k+1+j}u_{k+2}$, by $\ref{i->i+j}(a)$, it follows that $u_{k+2}u_{k+2+j}\in A(D)$. 
Continuing these two steps it follows that for $4\le j\le n-4$, if $u_ku_{k+j}\in A(D)$, then $u_iu_{i+j}\in A(D)$ for every  $u_i\in V(D)$.
\end{proofi}

If $u_0u_2\in A(D)$, then $u_3u_0\in A(D)$ and in this case $u_0u_4\in A(D)$. Continuing this argument we obtain that $u_0u_{2i}, u_{2i+1}u_0\in A(D)$. If the order of $D$ is equal to $n=2m$, then if $m$ is even $u_0u_{m}\in A(D)$ (if $m$ is odd, $u_{m}u_0\in A(D)$ resp.) and by Claim \ref{Claim} it follows that $u_iu_{i+m}\in A(D)$ ($u_{i+m}u_i\in A(D)$ resp.). In particular for $i=m$ we obtain that $u_{m}u_0\in A(D)$ ($u_0u_{m}\in A(D)$ resp.) implying that $\{u_0,u_{m}\}$ is a symmetric arc, a contradiction. Hence, the order of $D$ is odd. Let  $n=2m+1$. If $u_0u_2\in A(D)$, then by Claim \ref{Claim} it follows that $D\cong \overrightarrow{C}_{2m+1}(2,-3,\dots,(-1)^{m}(m))$. If $u_2u_0\in A(D)$, by a dual argument, we obtain  $D\cong \overrightarrow{C}_{2m+1}(-2,3,\dots,(-1)^{m-1}(m))$. Since $\overrightarrow{C}_{2m+1}(J)\cong \overrightarrow{C}_{2m+1}(-J)$, the result follows for in-semicomplete digraphs. 

The case when $D$ is a out-semicomplete digraph and $\overline{G}_D$ is an induced cycle on at least $6$ vertices is obtained by a dual argumentation, considering locally out-semicomplete digraphs instead of locally in-semicomplete digraphs and changing the orientation of all arcs in the proof. 
\end{proof}

We use the Strong Perfect Graph Theorem in \cite{chudnovsky2006strong} and Lemma \ref{semi ciclo} to characterize the locally in-semicomplete CKI-digraphs.

\begin{teo}\label{T in}
If $D$ is a  locally in-semicomplete CKI-digraph or a locally out-semicomplete CKI-digraph, then $D$ is a directed odd cycle, $D\cong \overrightarrow{C}_7(1,2)$ or $D\cong \overrightarrow{C}_n(1,\pm2,\pm3,\dots,\pm\lfloor \frac{n}{2}\rfloor)$, for some $n\ge4$.
\end{teo}

\begin{proof}
Let $D$ be a  locally in-semicomplete CKI-digraph or a  locally out-semicomplete CKI-digraph. By Theorem \ref{*}, if $G_D$ is a perfect graph, then $D\cong \overrightarrow{C}_3$ or $D\cong \overrightarrow{C}_n(1,\pm2,\pm3,\dots,\pm\lfloor \frac{n}{2}\rfloor)$ for some $n\ge4$.
Assume that $G_D$ is not a perfect graph. By Theorem \ref{T Chud}, $G_D$ has an induced odd cycle $C$ on at least $5$ vertices or $\overline{G}_D$ has an induced odd cycle $C$ on at least $7$ vertices.
If $G_D$ has an induced odd cycle $C$, then by Lemma \ref{semi ciclo}, $D[C]$ is a directed odd cycle, which is CKI. Let ${C=u_0u_1\dots u_{2n}}u_0$. By {Remark \ref{subNIC}}, $D[C]$ is not a proper subdigraph. Hence, $D$ is a directed odd cycle.
If $\overline{G}_D$ has an induced odd cycle $C$, then by  Lemma \ref{semi ciclo}, $D\cong \overrightarrow{C}_{2n+1}(2,-3,4,-5,\dots, (-1)^nn)$.

If $n>3$, then $2n+1\ge 9$. Let $i,j,k\ge3$ be odd integers such that $i+j+k=2n+1$. Since $u_{i+j+k}=u_0$, we have the  proper  asymmetric $\overrightarrow{C}_3=u_0u_iu_{i+j}u_{i+j+k}$, which  contradicts Remark \ref{sub CKI}, because $\overrightarrow{C}_3$ and $D$ are CKI-digraphs.
Since $C$ has at least $7$ vertices,  $n=3$ and  $D[V(C)]=\overrightarrow{C}_{7}(2,-3)$, which is isomorphic to $\overrightarrow{C}_{7}(1,2)$. 
By Remark \ref{CKI}, $\overrightarrow{C}_{7}(1,2)$ is CKI, then by Remark \ref{sub CKI}, $D[V(C)]$ is not a proper subdigraph and $D\cong \overrightarrow{C}_{7}(1,2)$. Hence, if  $\overline{G}_D$ has an induced odd cycle, then $D$ is $\overrightarrow{C}_7(1,2)$, which is CKI.

Thus, if $D$ is a  locally in-/out-semicomplete CKI-digraph, then $D$ is a directed odd cycle, $D\cong \overrightarrow{C}_7(1,2)$ or $D\cong \overrightarrow{C}_n(1,\pm2,\pm3,\dots,\pm\lfloor \frac{n}{2}\rfloor)$, for some $n\ge4$.
\end{proof}

Characterizing the locally in-semicomplete CKI-digraphs and the locally out-semicomplete CKI-digraphs is not the same problem, because there are digraphs such that $D$ is CKI but $D^{-1}$ is not CKI, see \cite{duchet1981note}.  Although, in Theorem \ref{T in}, it turned out that the digraphs that characterize the locally in-semicomplete CKI-digraphs and the locally out-semicomplete CKI-digraphs remains the same. 
Using the fact that  locally semicomplete digraphs are both locally in-semicomplete and locally out-semicomplete, we have the following result as a corollary of Theorem \ref{T in}.
\begin{coro}[\cite{galeana2016characterization}]\label{T}
If $D$ is a locally semicomplete CKI-digraph, then $D$ is a directed odd cycle, $D\cong \overrightarrow{C}_7(1,2)$ or $D\cong \overrightarrow{C}_n(1,\pm2,\pm3,\dots,\pm\lfloor \frac{n}{2}\rfloor)$, for some $n\ge4$.
\end{coro}

As a consequence of  Theorem \ref{T in} and Corollary  \ref{T}, it turns out that the characterizations of the CKI-digraphs of locally semicomplete digraphs and locally in-/out-semicomplete families are the same. Hence, we have the following result.
\begin{teo}\label{T in/out KP}
Locally in-semicomplete digraphs, locally out-semicomplete digraphs and locally semicomplete digraphs are KP if and only if they have no induced subdigraph $H$ such that $H$ is a directed odd cycle, $H\cong \overrightarrow{C}_7(1,2)$ or $H\cong \overrightarrow{C}_n(1,\pm2,\pm3,\dots,\pm\lfloor \frac{n}{2}\rfloor)$, for some $n\ge4$.
\end{teo}

%%%%%%%%%%%%%%%%%%%%%%%%%%%%%%%%%%%%%%%%%%%%%%%%%%%%%%%%%%%%

\section{Generalized bipartite tournaments}

In this section we characterize the arc-locally in-semicomplete, arc-locally out-semicomplete, $3$-quasi-transitive and $3$-anti-quasi-transitive  CKI-digraphs. 
The following result is a reformulation of the original result.
\begin{teo}[Theorem 4.3 \cite{galeana1984kernels}]\label{VH} Let $D$ be a CKI-digraph, which is not a directed odd cycle. For every vertex $u_0\in V(D)$ there is a directed cycle $C=u_0u_1\dots u_{2n}u_0$ such that $C$ has no diagonal $u_i u_j$ with $j\in\{0\}\cup \{1,3,\dots,2n-1\}$, $i\in \{0,1,\dots,2n\}$.
\end{teo}

\begin{teo}\label{VH min} Let $D$ be a CKI-digraph, which is not a directed odd cycle. For every vertex ${u_0\in V(D)}$ there exist a directed cycle $C=u_0u_1\dots u_{2n}u_0$ with $n\ge2$ such that
\begin{enumerate}
\item[$i)$] $C$ has a diagonal.
\item[$ii)$] $C$ has no diagonal $u_i u_j$ with $j\in\{0\}\cup \{1,3,\dots,2n-1\}$, $i\in \{0,1,\dots,2n\}$.
\item[$iii)$] $C$ has no diagonal $u_{2i-1}u_{2j}$ with $0<i<j\le n$. \label{subciclo}
\end{enumerate}
\end{teo}

\begin{proof}
$i)$ follows by Remark \ref{subNIC}. By Theorem \ref{VH}, for each vertex $u_0\in V(D)$ there is a directed cycle $C=u_0u_1\dots u_{2n}u_0$ such that $C$ has no diagonal $u_i u_j$ with $j\in\{0\}\cup \{1,3,\dots,2n-1\}$ and $i\in \{0,1,\dots,2n\}$. Hence, $ii)$ follows. Consider a directed cycle $C$ satisfying $ii)$ of minimum length.
If $C$ has a diagonal $u_{2i+1}u_{2j}$ with $0<i<j\le n$, then $C'=u_0u_1\dots u_{2i+1}u_{2j}\dots u_{2n}u_0$ is a cycle satisfying $ii)$ and $C'$ is shorter than $C$, contradicting the choice of $C$ and $iii)$ is proved. 
\end{proof}

A not necessarily directed path $P=uvwx$ is an $H_1$-path if $u\rightarrow v\leftarrow w\leftarrow x$; an $H_2$-path if  $u\leftarrow v\leftarrow w\rightarrow x$; an $H_3$-path if $P$ is a directed path and an $H_4$-path if $P$ is an anti-directed path.
For $i=1,2,3,4$, Bang-Jensen defined a digraph to be $H_i$\emph{-free}, if every $H_i$-path $uvwx$ has an arc between $u$ and $x$.
In \cite{bang2004structure} the arc-locally in-/out-semicomplete and the $3$-quasi-(anti-)transitive digraphs where  defined in terms of $H_i$-free digraphs, $i=1,2,3,4$.
The family of $H_1$-free digraphs ($H_2$-free digraphs, resp.) is the family of arc-locally in-semicomplete digraphs (arc-locally out-semicomplete digraphs, resp.).
The $H_3$-free digraphs are the $3$-quasi-transitive digraphs and
the $H_4$-free digraphs are the $3$-anti-quasi-transitive digraphs. We denote by $TT_3$ the transitive (acyclic) tournament on $3$ vertices.

\begin{teo}\label{T1}
The unique asymmetric $3$-quasi-transitive CKI-digraph is $\overrightarrow{C}_3$. The directed odd cycles are the only asymmetric CKI-digraphs which are arc-locally in-semicomplete digraphs, arc-locally out-semicomplete digraphs or $3$-anti-quasi-transitive $TT_3$-free digraphs.
\end{teo}

\begin{proof}
The asymmetric $3$-quasi-transitive digraph $\overrightarrow{C}_3$ is CKI.
Suppose, for a contradiction, that  $D$ is an asymmetric $3$-quasi-transitive CKI-digraph, which is not $\overrightarrow{C}_3$.
By hypothesis $D$ is $H_3$-free and since directed odd cycles of order at least $5$ have induced directed paths of order $4$, $D$ is not an odd cycle. 
Let $u_0\in V(D)$. By Theorem \ref{VH min}, $D$ has a directed odd cycle of minimum length $C=u_0u_1\dots u_{2n}u_0$, $n\ge2$, such that $C$ has no diagonal $u_i u_j$ with $i\in\{0,1,\dots,2n\}$ and $j\in\{0\}\cup\{1,3,5,\dots,2n-1\}$.
Since $D$ is CKI, $C$ is not induced and $C$ has at least five vertices. By definition, the $H_3$-path $u_{2n-1}u_{2n}u_0u_1$ must have a diagonal between  the vertices $u_{2n-1}$ and $u_1$, that is a diagonal between two vertices with odd subindices, contradicting the choice of $C$ in Theorem \ref{VH min} $ii)$. Hence, the unique asymmetric $3$-quasi-transitive CKI-digraphs is $\overrightarrow{C}_3$.

Directed odd cycles are asymmetric arc-locally in-semicomplete (arc-locally out-semicomplete resp.) [$3$-anti-quasi-transitive $TT_3$-free resp.] CKI-digraphs.
Suppose, for a contradiction, that $D$ is an asymmetric arc-locally in-semicomplete (arc-locally out-semicomplete resp.) [$3$-anti-quasi-transitive $TT_3$-free resp.]  CKI-digraph, which is not a directed odd cycle. Let $u_0\in V(D)$. By Theorem \ref{VH min}, $D$ has a directed odd cycle of minimum length $C=u_0u_1\dots u_{2n}u_0$, $n\ge2$, such that $C$ has no diagonal $u_i u_j$ with $i\in\{0,1,\dots,2n\}$ and $j\in\{0\}\cup\{1,3,5,\dots,2n-1\}$. Observe that $C$ has at least five vertices. Let  $vw$ be a diagonal, with $v,w\in \{u_0,u_1,\dots, u_{2n}\}$ and the index are taken modulo $2n+1$.

We have two cases: $(i)$ The cycle $C$ and its diagonal $vw$ induce a directed cycle containing $u_0u_1$ or $(ii)$ the cycle $C$ and its diagonal $vw$ induce a directed cycle avoiding  $u_0u_1$.

$(i)$ The directed cycle $C$ and its diagonal $vw$ induce a directed cycle containing the arc $u_0u_1$. \\
By Theorem \ref{VH min} $ii)$ and $iii)$, the diagonal is $vw$=$u_{2i}u_{2j}$ with $0< i<j\le n$ ($0< i<j\le n$ resp.) [$1< i+1<j\le n$ resp.]. By definition, the $H_1$-path $u_{2i-1}u_{2i}u_{2j}u_{2j-1}$ must have a diagonal between the vertices  $u_{2i-1}$ and $u_{2j-1}$ (the $H_2$-path $u_{2i+1}u_{2i}u_{2j}u_{2j+1}$ must have a diagonal between the vertices  $u_{2i+1}$ and $u_{2j+1}$ resp.) [the $H_4$-path $u_{2j-1}u_{2j}u_{2i}u_{2i+1}$ must have a diagonal between the vertices  $u_{2i+1}$ and $u_{2j-1}$ resp., ($2i+1<2j-1$ because $D$ is $TT_3$-free)], in each case $vw$ is a diagonal between two vertices with  subindices which are odd or equal to zero (if $j=n$, then $2j+1\equiv 0~(mod~ 2n+1)$), contradicting the choice of $C$ in Theorem \ref{VH min} $ii)$. Hence, $C$ has no diagonal inducing a directed cycle containing the arc $u_0u_1$.

$(ii)$ The directed cycle $C$ and its diagonal $vw$ induce a directed cycle avoiding the arc $u_0u_1$. \\
By the choice of $C$, we may assume that $vw=u_{i}u_{2j}$ with $i=0$ and $1\le j< n$ ($1\le j< n$ resp.) [since $D$ is $TT_3$-free, $1< j< n$ resp.] or $i\neq0$ and $2j+2< i\le 2n$ (by Remark \ref{subNIC}, $D$ has no $\overrightarrow{C}_3$).

If $i=0$, since $vw$ is a diagonal, it follows that $1\le j< n$ ($1\le j< n$ resp.) [since $D$ is $TT_3$-free, $1< j< n$ resp.].
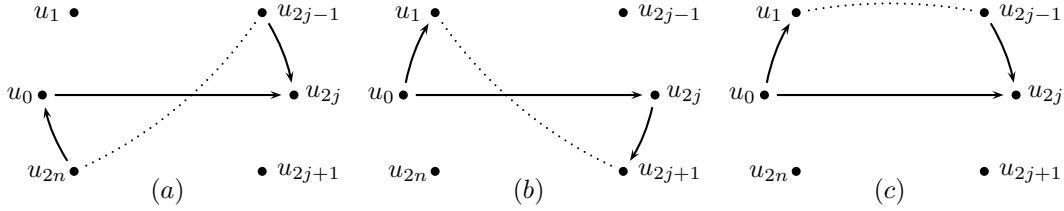
\begin{figure}[h!]
\centering
\begin{pspicture}(14,2.5)
        \psset{unit=.8, nodesep=3pt}        
        \rput(4.95,2.65){\rnode{1}{$u_{2j-1}$}}
        \rput(5.2,1.3){\rnode{2}{$u_{2j}$}}
        \rput(4.95,0){\rnode{3}{$u_{2j+1}$}}
        \rput(.6,0){\rnode{6}{$u_{2n}$}}
        \rput(0.15,1.3){\rnode{6}{$u_{0}$}}
        \rput(.6,2.65){\rnode{6}{$u_{1}$}}
        \rput(2.6,-.3){\rnode{6}{$(a)$}}
        \cnode*(4.700,1.3){1.7pt}{i}
        \cnode*(4.172,2.67){1.7pt}{i-}
        \cnode*(4.172,0.03){1.7pt}{i+}
        \cnode*(1.046,0.03){1.7pt}{j-}
        \cnode*(0.518,1.3){1.7pt}{j}
        \cnode*(1.046,2.67){1.7pt}{j+}
        \ncarc[arcangle=10]{->}{i-}{i}
        \ncline{<-}{i}{j}
        \ncarc[arcangle=10]{->}{j-}{j}
        \ncarc[arcangle=-15, linestyle=dotted,dotsep=2pt]{-}{j-}{i-}
        \rput(10.95,2.65){\rnode{1}{$u_{2j-1}$}}
        \rput(11.2,1.3){\rnode{2}{$u_{2j}$}}
        \rput(10.95,0){\rnode{3}{$u_{2j+1}$}}
        \rput(6.6,0){\rnode{6}{$u_{2n}$}}
        \rput(6.15,1.3){\rnode{6}{$u_{0}$}}
        \rput(6.6,2.65){\rnode{6}{$u_{1}$}}
        \rput(8.6,-.3){\rnode{6}{$(b)$}}
        \cnode*(10.700,1.3){1.7pt}{i}
        \cnode*(10.172,2.67){1.7pt}{i-}
        \cnode*(10.172,0.03){1.7pt}{i+}
        \cnode*(7.046,0.03){1.7pt}{j-}
        \cnode*(6.518,1.3){1.7pt}{j}
        \cnode*(7.046,2.67){1.7pt}{j+}
        \ncarc[arcangle=10]{->}{i}{i+}
        \ncline{<-}{i}{j}
        \ncarc[arcangle=15, linestyle=dotted,dotsep=2pt]{-}{i+}{j+}
        \ncarc[arcangle=10]{->}{j}{j+}
        \rput(16.95,2.65){\rnode{1}{$u_{2j-1}$}}
        \rput(17.2,1.3){\rnode{2}{$u_{2j}$}}
        \rput(16.95,0){\rnode{3}{$u_{2j+1}$}}
        \rput(12.6,0){\rnode{6}{$u_{2n}$}}
        \rput(12.15,1.3){\rnode{6}{$u_{0}$}}
        \rput(12.6,2.65){\rnode{6}{$u_{1}$}}
        \rput(14.6,-.3){\rnode{6}{$(c)$}}
        \cnode*(16.700,1.3){1.7pt}{i}
        \cnode*(16.172,2.67){1.7pt}{i-}
        \cnode*(16.172,0.03){1.7pt}{i+}
        \cnode*(13.046,0.03){1.7pt}{j-}
        \cnode*(12.518,1.3){1.7pt}{j}
        \cnode*(13.046,2.67){1.7pt}{j+}
        \ncarc[arcangle=10]{->}{i-}{i}
        \ncline{<-}{i}{j}
        \ncarc[arcangle=-10, linestyle=dotted,dotsep=2pt]{-}{i-}{j+}
        \ncarc[arcangle=10]{->}{j}{j+}
\end{pspicture}
\caption{The dotted arcs are not diagonals of $C$.}
\label{Fig 3}
\end{figure}
By definition, the $H_1$-path $u_{2n}u_{0}u_{2j}u_{2j-1}$ must have a diagonal between the vertices $u_{2n}$ and $u_{2j-1}$, since $2j-1$ is odd, by Theorem \ref{VH min} $ii)$, $u_{2j-1}u_{2n}$ is a diagonal of $C$, contradicting the minimality of $C$ in Theorem \ref{VH min} $iii)$, see Figure \ref{Fig 3}$(a)$ (the $H_2$-path $u_{1}u_{0}u_{2j}u_{2j+1}$ must have a diagonal between the vertices $u_{2j+1}$ and $u_{1}$, a diagonal between two vertices with odd subindices, a contradiction to the choice of $C$ in Theorem \ref{VH min} $ii)$, see Figure \ref{Fig 3}$(b)$  resp.) [the anti-directed path $u_{2j-1}u_{2j}u_{0}u_{1}$ must have a diagonal between the vertices $u_{2j-1}$ and $u_{1}$ ($u_{2j-1}\neq u_{1}$ because $D$ is $TT_3$-free), a diagonal between two vertices with odd subindices, contradicting the choice of $C$ in Theorem \ref{VH min} $ii)$, see Figure \ref{Fig 3}$(c)$ resp.]. 

Hence,  $i\neq 0$ and since, $j\ge1$, it follows that $4\le2j+2< i\le 2n$. 
If $i$ is odd, by definition, the $H_1$-path $u_{i-1}u_{i}u_{2j}u_{2j-1}$ must have a diagonal between the vertices  $u_{i-1}$ and $u_{2j-1}$ (the $H_2$-path $u_{i+1}u_{i}u_{2j}u_{2j+1}$ must have a diagonal between the vertices  $u_{2j+1}$ and $u_{i+1}$ resp.) [the anti-directed path $u_{2j-1}u_{2j}u_{i}u_{i+1}$ must have a diagonal between the vertices $u_{2j-1}$ and $u_{i+1}$ ($u_{2j-1}\neq u_{1}$ because $D$ is $TT_3$-free) resp.]. By  Theorem \ref{VH min} $ii)$, the diagonal must be  $u_{2j-1}u_{i-1}$ ($u_{2j+1}u_{i+1}$) [$u_{2j-1}u_{i+1}$ resp.] contradicting the minimality of the cycle $C$ in Theorem \ref{VH min} $iii)$.
Hence, $i$ is even. By definition, the $H_1$-path $u_{i-1}u_{i}u_{2j}u_{2j-1}$ must have a diagonal between the vertices  $u_{i-1}$ and $u_{2j-1}$ contradicting the choice of $C$ in Theorem \ref{VH min} $ii)$ because $i-1$ is odd (the $H_2$-path $u_{i+1}u_{i}u_{2j}u_{2j+1}$ must have a diagonal between the vertices  $u_{2j+1}$ and $u_{i+1}$ contradicting the choice of $C$ in Theorem \ref{VH min} $ii)$ because $i+1$ is odd or $i+1=0$ resp.)
[the anti-directed path $u_{2j-1}u_{2j}u_{i}u_{i+1}$ must have a diagonal between the vertices $u_{2j-1}$ and $u_{i+1}$. Since $i+1$ is odd or $i+1=0$, it follows that a diagonal between the vertices $u_{2j-1}$ and $u_{i+1}$ contradicts the choice of $C$ in Theorem \ref{VH min} $ii)$, unless the diagonal is the arc  $u_0u_{1}$ ($i=2n$, $i+1=0$ and $2j-1=1$), because in this case the diagonal of the anti-directed path $u_{2j-1}u_{2j}u_{i}u_{i+1}$ is not a diagonal of the directed cycle $C$. By Remark \ref{subNIC}, the odd cycle $C'=u_2u_3\dots u_{2i}u_2$ is not an induced cycle in $D$ and since any diagonal of $C'$ is a diagonal of $C$, by the choice of $C$, $C'$ has a diagonal $v'w'=u_{i'}u_{2j'}$ such that $4< 2j'+2<i'\le 2n$ or $4\le 2j'+2<i'< 2n$. Since, $v'w'$ is a diagonal of $C$, the anti-directed path $u_{2j'-1}u_{2j'}u_{i'}u_{i'+1}$ must have a diagonal between the vertices $u_{2j'-1}$ and $u_{i'+1}$ contradicting the choice of $C$ in Theorem \ref{VH min} $ii)$ because $3\le2j'-1<2n$, and  $i'+1=0$ or $i'+1$ is odd resp.].

Both cases lead to a contradiction, thus $D$ is an odd cycle.
\end{proof}

As a consequence of  Theorem \ref{T1} we have the following result.
 
\begin{teo}\label{T1 KP}
An asymmetric $3$-quasi-transitive digraph is KP if and only if it has no induced $\overrightarrow{C}_3$.
Asymmetric arc-locally in-/out-semicomplete  digraphs and $3$-anti-quasi-transitive $TT_3$-free digraphs  are KP if and only if they have no induced directed odd cycle.
\end{teo}

%%%%%%%%%%%%%%%%%%%%%%%%%

\section{Conclusions}
For tournaments, there is a unique CKI-digraph namely the $\overrightarrow{C}_3$, and for each integer $n$ the transitive tournament of order $n$ is the unique KP-digraph.
The  families of generalized tournaments considered in this paper have a nice characterization of their CKI-digraphs due to Corollary  \ref{C semi quasi} and Theorem \ref{T in}, 
and a nice characterization of their KP-digraphs due to Theorem \ref{*KP} and Theorem \ref{T in/out KP}. 
Hence, these families of generalized tournaments somehow preserve the property of the tournaments, that the characterizations of the CKI- and the KP-digraphs are nice.
 
All bipartite tournaments are KP-digraphs, so there are no bipartite tournaments which are CKI-digraphs. 
The  families of generalized bipartite tournaments considered in this paper have a nice characterization of their CKI-digraphs due to 
Theorem \ref{T1} and a nice characterization of their KP-digraphs due to Theorem \ref{T1 KP}. 
Hence, these families of generalized bipartite tournaments somehow preserve the property of the bipartite tournaments, that there is only one class of CKI-digraphs, although not every digraph is KP.

The \emph{asymmetric part} of a digraph $D$ is the spanning subdigraph of $D$ induced by the asymmetric arcs of $D$. 
The asymmetric part of the digraphs considered in this paper is either isomorpic to $\overrightarrow{C}_7(1,2)$ or isomorphic to a cycle. 
As pointed out in \cite{galeana2016characterization}, deciding whether a digraph $D$ is isomorphic to a directed cycle or isomorphic to $\overrightarrow{C}_7(1,2)$ is polynomial, and for the case when the asymmetric part of the digraph $D$ is a cycle, deciding whether the digraph $D$ is an odd cycle or the digraph $D$ is isomorphic to $\overrightarrow{C}_n(1,\pm2,\pm3,\dots,\pm\lfloor \frac{n}{2}\rfloor)$ for some $n\ge4$ is also polynomial. Therefore, it is polynomial to determine whether a digraph is CKI if the underlying graph is a perfect graph or if the digraph is semicomplete, semicomplete multipartite, quasi-transitive, locally in-/out-semicomplete, locally  semicomplete, asymmetric $3$-quasi-transitive, asymmetric arc-locally in-/out-semicomplete and $3$-anti-quasi-transitive $TT_3$-free.

\section{Acknowledgements}
We thank the anonymous referees for their comments and suggestions that helped us to significantly improve the presentation of this paper.
%%%%%%%%%%%%%%%%%%%%%%%%

%\section*{References}
\nocite{*}
\bibliographystyle{abbrvnat}
% use the following instead if you encounter problems 
%\bibliographystyle{alpha}
\bibliography{KP-local}
\label{sec:biblio}

\end{document}